\documentclass[12pt]{amsart}
\usepackage[margin = 1in]{geometry}
\usepackage{amsmath,amssymb,amsthm}
\usepackage[dvipsnames]{xcolor}
\usepackage{tikz}

\newtheorem{theorem}{Theorem}
\newtheorem{theorem'}{Theorem*}

\newtheorem{proposition}[theorem]{Proposition}

\theoremstyle{definition}
\newtheorem{definition}[theorem]{Definition}

\title{Computing Galois Groups of Fano Problems}
\author[T.~Yahl]{Thomas Yahl} 
\address{T.~Yahl\\ 
         Department of Mathematics\\ 
         Texas A\&M University\\ 
         College Station\\ 
         Texas \ 77843\\ 
         USA} 
\email{thomasjyahl@math.tamu.edu} 
\urladdr{http://www.math.tamu.edu/~thomasjyahl} 
\newcommand{\gr}{\mathbb{G}}
\newcommand{\gal}{\mathcal{G}}

\definecolor{myBlue}{rgb}{0, 0.22, 0.66}
\newcommand{\defcolor}[1]{{\color{myBlue}#1}}

\begin{document}

%%%%%%%%%%%%
%%Abstract%%
%%%%%%%%%%%%
\begin{abstract}
A Fano problem consists of enumerating linear spaces of a fixed dimension on a variety, generalizing the classical problem of 27 lines on a cubic surface. Those Fano problems with finitely many linear spaces have an associated Galois group that acts on these linear spaces and controls the complexity of computing them in suitable coordinates. These Galois groups were first defined and studied by Jordan, who in particular considered the problem of lines in $\mathbb{P}^3$ on a cubic surface. Recently, Hashimoto and Kadets determined the Galois groups for a special family of Fano problems and showed that all other Fano problems have Galois group that contains the alternating group. A complete classification of Galois groups of Fano problems rests on the determination of these Galois groups. We use computational tools to prove that several Fano problems of moderate size have Galois group equal to the symmetric group, each of which were previously unknown.
\end{abstract}
%%%%%%%%%%%%%%%%%%%%%%%%%%%%%%%%%%%%%%%%%%%%%%%%%%%%%%%%%%%%%%%%%%%%%%%%%%%%%%%%%

\maketitle

%%%%%%%%%%%%%%%%
%%Introduction%%
%%%%%%%%%%%%%%%%
\section{Introduction}
The classical problem from enumerative geometry to describe the set of 27 lines in $\mathbb{P}^3$ on a smooth cubic surface is one of the first examples of a Fano problem: enumerating $r$--planes in $\mathbb{P}^n$ lying on a variety $X$. The Grassmanian is a complex projective varieties that parameterizes the set of all $r$--planes in $\mathbb{P}^n$, and the family of $r$--planes that lie on $X$ is a subscheme of the Grassmanian called the Fano scheme of $X$. We consider $X$ a general complete intersection, where invariants such as dimension and degree of its Fano scheme are determined using combinatorial data. When the Fano scheme of $X$ is finite, we call the problem of describing its Fano scheme a Fano problem. This setting covers familiar problems such as lines in $\mathbb{P}^3$ on a cubic surface and lines in $\mathbb{P}^4$ on the intersection of two quadric hypersurfaces.

To each Fano problem there is an associated Galois group which acts on the Fano scheme of $X$. Jordan was the first to study these Galois groups in his work  ``Trait\'{e} des Substitutions et des \'{E}quations Alg\'{e}briques" in which he noted that the Galois group of an enumerative problem must preserve intrinsic structure of the problem \cite{Jordan}. Jordan further instituted the idea that the Galois group of an enumerative problem must be as large as these structures permit. For instance, the Galois group of the problem of lines in $\mathbb{P}^3$ on a cubic surface must preserve the incidences among these lines. From this, Jordan observed this Galois group is a subgroup of the Coxeter group $E_6$, and it has since been shown that this Galois group is equal to $E_6$ by Harris.

Harris observed that the algebraic Galois groups Jordan defined are geometric monodromy groups, an idea tracing back to Hermite \cite{Hermite}. Using this, he generalized Jordan's work by studying the Galois group of the problem of lines in $\mathbb{P}^n$ on a hypersurface of degree $2n-3$. After showing Jordan's inclusion to be an equality (the case $n=3$), he showed that for $n\ge 4$ the Galois group is the symmetric group acting on a general Fano scheme---such a Galois group is called fully symmetric. To show this, Harris proved these Galois groups are two--transitive and showed they contain a simple transposition. This simple transposition is the result of an explicitly constructed point whose local monodromy yields a simple transposition. We use Harris' technique of producing a simple transposition to determine Galois groups of other Fano problems.

The study of Galois groups of Fano problems then laid dormant until Hashimoto and Kadets nearly determined them in all cases \cite{HK}. They first proved the Fano problem of $r$--planes in $\mathbb{P}^{2r+2}$ on the intersection of two quadrics has Galois group equal to the Coxeter group $D_{2r+3}$ for $r\ge 1$. Then it was shown that these Fano problems and the problem of lines on a cubic surface are enriched in the sense that these are the only Fano problems where the $r$--planes of the Fano scheme intersect. Using this, it was shows that all other Fano problems have Galois group that is highly transitive and contains the alternating group---such a Galois group is said to be at least alternating.

By the results of Hashimoto and Kadets, the open problem of classifying Galois groups of Fano problems rests on determining whether those at least alternating Galois groups of Fano problems are equal to the alternating group or the symmetric group. Towards this goal, we use computational tools to prove Galois groups of Fano problems of moderate size are fully symmetric. We do so by using Harris' technique of exhibiting a simple transposition for these Fano problems. This transposition is the result of producing a system satisfying certain properties and we verify those properties with a mixture of exact computation and numerical certification. 

\begin{center}
\begin{figure}[h]
  \includegraphics[scale=.38]{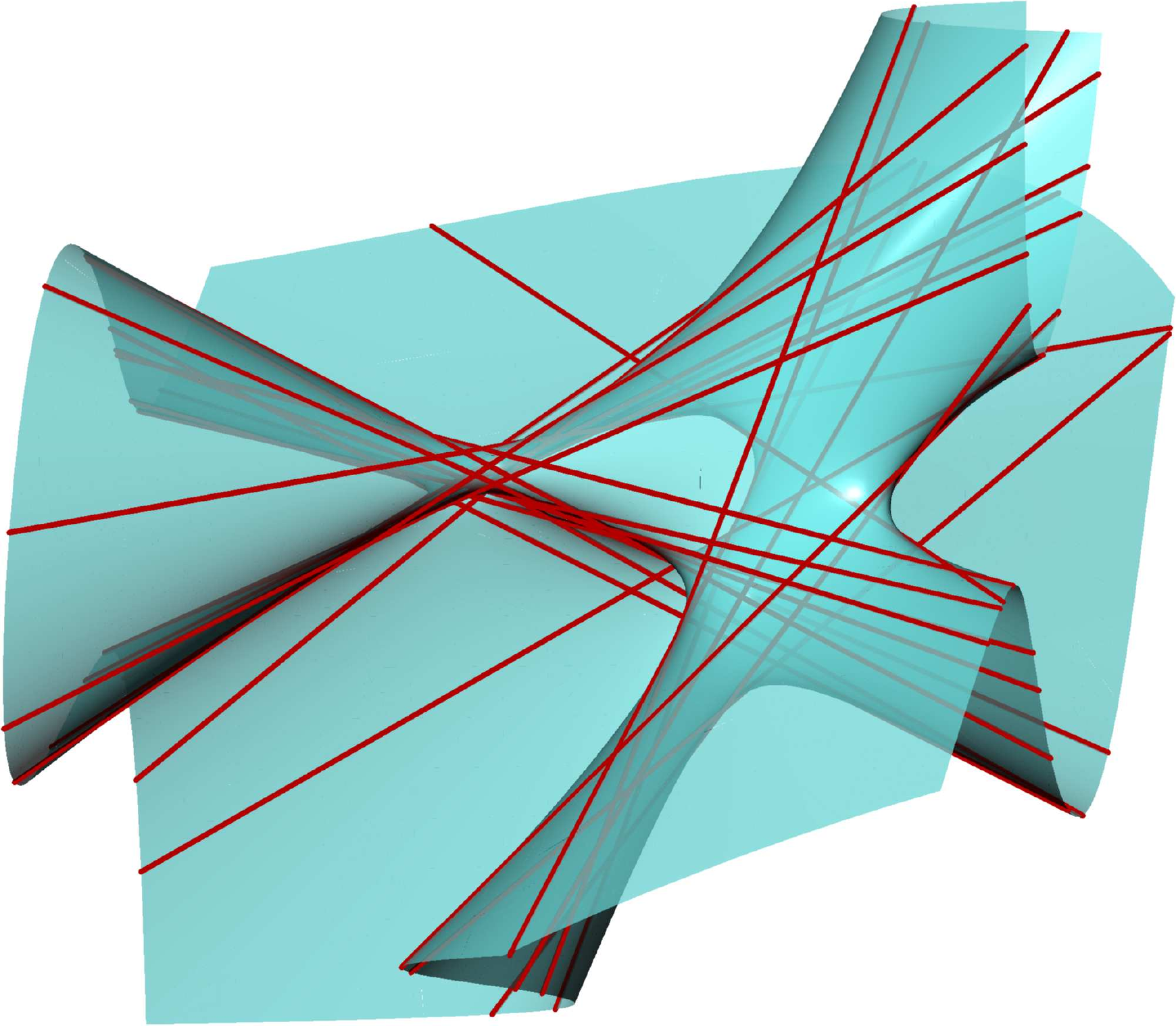}
  \caption{27 lines on a smooth cubic surface}
\end{figure}
\end{center}

%%%%%%%%%%%%%%%%%
%%Fano Problems%%
%%%%%%%%%%%%%%%%%
\section{Fano Problems}
The family of $r$--planes in $\mathbb{P}^n$ is a complex irreducible projective variety known as the Grassmanian $\gr(r,\mathbb{P}^n)$ and is of dimension $(r+1)(n-r)$. For a variety $X\subseteq\mathbb{P}^n$ its \defcolor{Fano scheme} is the subscheme of $\gr(r,\mathbb{P}^n)$ of $r$--planes that lie on $X$. We study Fano schemes uniformly in the setting that $X$ is a general complete intersection.

When $X\subseteq\mathbb{P}^n$ is a codimension $s$ subvariety, $X$ is the zero locus of homogeneous polynomials $F = (f_1,\dotsc,f_s)$ in $n+1$ variables. Write $d_\bullet = (d_1,\dotsc,d_s)$ for the sequence of respective degrees of these polynomials, $\deg f_i = d_i$. We make the assumption that $d_i\ge 2$ for $i=1,\dotsc,s$, since otherwise some $f_i$ is linear from which we may consider $X\subseteq\mathbb{P}^{n-1}$. Since a general complete intersection $X$ is smooth, it contains no $r$--planes of dimension $r>\frac{1}{2}\dim X$---so we require that $2r\le n-s$ from this point.

Let $\mathbb{C}^{(r,n,d_\bullet)}$ denote the space of systems of complex homogeneous polynomials $F = (f_1,\dotsc,f_s)$ in $n+1$ variables and of respective degrees $d_\bullet = (d_1,\dotsc,d_s)$. The zero set of a general system $F\in\mathbb{C}^{(r,n,d_\bullet)}$ defines a smooth complete intersection in $\mathbb{P}^n$, whose Fano scheme of $r$--planes we consider. For a system $F\in\mathbb{C}^{(r,n,d_\bullet)}$, write $\mathcal{V}_r(F)\subseteq\mathbb{G}(r,\mathbb{P}^n)$ for the Fano scheme of $r$--planes on its zero set in $\mathbb{P}^n$. We say a Fano scheme is general if it is the Fano scheme of a general system $F\in\mathbb{C}^{(r,n,d_\bullet)}$. Invariants such as dimension and degree of general Fano schemes are determined by the combinatorial data $(r,n,d_\bullet)$, called the \defcolor{type} of the Fano scheme.

%Dimension & degree
We determine the expected dimension of a Fano scheme of a given type $(r,n,d_\bullet)$. Fix $F = (f_1,\dotsc,f_s)\in\mathbb{C}^{(r,n,d_\bullet)}$. We note that $\ell\in\mathcal{V}_r(F)$ exactly when $F|_\ell = 0$, or equivalently $f_i|_\ell=0$ for $i=1,\dotsc,s$. Each restriction $f_i|_\ell$ is a degree $d_i$ form on $\ell$, and the space of such forms has dimension $\left(\begin{smallmatrix}d_i+r\\r\end{smallmatrix}\right)$. As such, we expect the vanishing of these forms to have codimension $\sum_{i=1}^s \left(\begin{smallmatrix}d_i+r\\r\end{smallmatrix}\right)$ and the dimension of $\mathcal{V}_r(F)$ to be
\begin{align*}
\defcolor{\delta(r,n,d_\bullet)} = (r+1)(n-r) - \sum_{i=1}^s\begin{pmatrix}d_i+r\\r\end{pmatrix}.
\end{align*}
There is an alternative view of this description of the Fano scheme $\mathcal{V}_r(F)$. Each $F\in\mathbb{C}^{(r,n,d_\bullet)}$ determines a section of a certain vector bundle (this is a direct sum of symmetric powers of the dual of the tautological bundle on $\mathbb{G}(r,\mathbb{P}^n)$) and the Fano scheme $\mathcal{V}_r(F)$ is the vanishing locus of this section. This vector bundle has rank $\sum_{i=1}^s \left(\begin{smallmatrix}d_i+r\\r\end{smallmatrix}\right)$ and thus we arrive at the same expected dimension. Debarre and Manivel \cite{DM} take this view to show this expected dimension is the dimension for general Fano schemes.

\begin{theorem}[Debarre,Manivel]
The general Fano scheme of type $(r,n,d_\bullet)$ is non-empty, smooth, and has dimension $\delta(r,n,d_\bullet)$ when $\delta(r,n,d_\bullet)\ge 0$ and is empty otherwise.
\end{theorem}

We say the \defcolor{Fano problem} $(r,n,d_\bullet)$ is the problem of describing a general Fano scheme of type $(r,n,d_\bullet)$, where $\delta(r,n,d_\bullet)=0$. That is, a Fano problem is that of describing a finite general Fano scheme.

Since the Grassmanian is a projective variety, a Fano scheme has a well-defined degree. For a fixed tuple $(r,n,d_\bullet)$, there is a generic degree given as the maximal degree of a Fano scheme of this type. For a Fano problem $(r,n,d_\bullet)$, this degree is written \defcolor{$\deg(r,n,d_\bullet)$} and is the cardinality of a general Fano scheme of type $(r,n,d_\bullet)$. By describing the Fano scheme $\mathcal{V}_r(F)$ as the vanishing locus of a section of a certain vector bundle, the degree of a Fano problem can be computed by a chern class computation. Debarre and Manivel explicitly computed this quantity using the splitting lemma and properties of Schur polynomials. 

Define the quantities
\begin{align*}
Q_{r,d}(x) = \prod_{\stackrel{a_i\in\mathbb{Z}_{\ge 0}}{a_0+\dotsb+a_r=d}}(a_0x_0 + \dotsb + a_rx_r)\in\mathbb{Z}[x_0,\dotsb,x_r]
\end{align*}
and $Q_{r,d_\bullet}(x) = Q_{r,d_0}(x)\dotsb Q_{r,d_s}(x)$, as well as the Vandermonde polynomial
\begin{align*}
V(x) = \prod_{0\le i<j\le r}(x_i-x_j).
\end{align*}

\begin{theorem}[Debarre,Manivel]
The degree of the Fano problem $(r,n,d_\bullet)$, $\deg(r,n,d_\bullet)$, is equal to the coefficient of $x_0^n x_1^{n-1}\dotsb x_r^{n-r}$ in the product $Q_{r,d_\bullet}(x)V(x)$. 
\end{theorem}

When a Fano scheme $\mathcal{V}_r(F)$ of a Fano problem is finite (as is expected), there are $\deg(r,n,d_\bullet)$ many points of the Fano scheme counting multiplicity. Table \ref{Small Fano} shows some Fano problems of small degree. Note the first two rows are the familiar problem of 16 lines in $\mathbb{P}^4$ on the intersection of two quadric hypersurfaces and the problem of 27 lines in $\mathbb{P}^3$ on a cubic hypersurface.

\begin{table}[htb]
  \caption{Fano problems of small degree}
  \label{Small Fano}
  \def\arraystretch{1.2}
  \begin{tabular}{||c|c|c|c|c||}
    \hline
    $~r~$ & $~n~$ & $~d_\bullet~$ & $~\deg(r,n,d_\bullet)~$ & ~Galois Group~\\
    \hline\hline
    1 & 4 & $(2,2)$ & 16 & $D_5$\\
    \hline
    1 & 3 & $(3)$ & 27 & $E_6$\\
    \hline
    2 & 6 & $(2,2)$ & 64 & $D_7$\\
    \hline
    3 & 8 & $(2,2)$ & 256 & $D_9$\\
    \hline
    1 & 7 & $(2,2,2,2)$ & 512 & $S_{512}$\\
    \hline
    1 & 6 & $(2,2,3)$ & 720  & $S_{720}$\\
    \hline
    4 & 10 & $(2,2)$ & 1024 & $D_{11}$\\
    \hline
    2 & 8 & $(2,2,2)$ & 1024  & $S_{1024}$\\
    \hline
  \end{tabular}
\end{table}

Lower bounds exist for $\deg(r,n,d_\bullet)$, which can then be used to enumerate Fano problems with degree less than a given amount. For instance, we use the following lower bounds on the degree of a Fano problem to enumerate Fano problems.

\begin{proposition}
Let $(r,n,d_\bullet)$ be a Fano problem. We have
\begin{align*}
\deg(r,n,d_\bullet) \ge \prod_{1\le i\le s\,} \prod_{\stackrel{1\le j\le r+1}{j\mid d_i}} \left(\frac{d_i}{j}\right)^{\left(\begin{smallmatrix}r+1\\j\end{smallmatrix}\right)} \ge \prod_{1\le i\le s} d_i^{r+1}.
\end{align*}
\end{proposition}
\begin{proof}
For the first inequality, consider the product expansion of $Q_{r,d_i}(x)$ for fixed $i$. For each $1\le j\le r+1$ such that $j\mid d_i$, there are $\left(\begin{smallmatrix}r+1\\j\end{smallmatrix}\right)$ many terms of the product with $j$ of the $a_k$ equal to $d_i/j$. The quantity $d_i/j$ then factors out of each of these ${\left(\begin{smallmatrix}r+1\\j\end{smallmatrix}\right)}$ many terms and appears as a constant factor of $Q_{r,d_i}(x)$ with this respective power. Hence this is a factor of every term of $Q_{r,d_\bullet}(x)V(x)$ and of $\deg(r,n,d_\bullet)$. 
\end{proof}

%%%%%%%%%%%%%%%%%%%%%%%%%%%%%%%%%%
%%Galois Groups of Fano Problems%%
%%%%%%%%%%%%%%%%%%%%%%%%%%%%%%%%%%
\section{Galois Groups of Fano Problems}
A Fano problem $(r,n,d_\bullet)$ determines an incidence correspondence.
\begin{center}
\begin{tikzpicture}
\node at (0,1.5) {$\Gamma = \left\{ (F,\ell)\in\mathbb{C}^{(r,n,d_\bullet)}\times\gr(r,\mathbb{P}^n):F|_\ell=0 \right\}$};
\draw[->] (-3.72,1.2)--(-3.72,.3) node [left] at (-3.66,.75) {\defcolor{$\pi_{(r,n,d_\bullet)}$}};
\node at (-3.68,0) {$\mathbb{C}^{(r,n,d_\bullet)}$};
\draw[->] (4.05,1.5)--(5.05,1.5);
%\node [above] at (4.55,1.5) {$\rho$};
\node at (6,1.5) {$\mathbb{G}(r,\mathbb{P}^n)$};
\end{tikzpicture}
\end{center}
The map $\rho$ realizes the incidence variety $\Gamma$ as a vector bundle over $\mathbb{G}(r,\mathbb{P}^n)$. Indeed, the fiber over $\ell\in\mathbb{G}(r,\mathbb{P}^n)$ is the kernel of the surjective linear map restricting homogeneous forms $F = (f_1,\dotsc,f_s)$ of respective degrees $d_\bullet$ to homogeneous forms on $\ell$ of the same respective degrees. The codimension of these fibers in $\mathbb{C}^{(r,n,d_\bullet)}$ is then equal to $\sum_{i=1}^s \left(\begin{smallmatrix}d_i+r\\r\end{smallmatrix}\right)=(r+1)(n-r)$; it follows that $\Gamma$ is smooth and of dimension $\dim \mathbb{C}^{(r,n,d_\bullet)}$. 

Given $F\in\mathbb{C}^{(r,n,d_\bullet)}$, the fiber $\pi_{(r,n,d_\bullet)}^{-1}(F)$ is the Fano scheme $\mathcal{V}_r(F)$. By the results of Debarre and Manivel, there is a Zariski open set $U\subseteq\mathbb{C}^{(r,n,d_\bullet)}$ (in particular dense, open, and path-connected) with the property that if $F\in U$, the fiber $\pi_{(r,n,d_\bullet)}^{-1}(F)$ consists of $\deg(r,n,d_\bullet)$ smooth points. It follows that the restriction of $\pi_{(r,n,d_\bullet)}$ to $\pi_{(r,n,d_\bullet)}^{-1}(U)$ is a smooth covering space of degree $\deg(r,n,d_\bullet)$. 

Fix a base point $F\in U$. Every directed loop in $U$ based at $F$ lifts to $\deg(r,n,d_\bullet)$ paths in $\Gamma$ each starting at distinct poitns of the fiber $\pi_{(r,n,d_\bullet)}^{-1}(F)$; the endpoints of these paths determine a permutation of this fiber. The set of all permutations obtained this way is the monodromy group of $\pi_{(r,n,d_\bullet)}$ and is defined up to isomorphism for different choices of base point $F\in U$ and reordering of the fiber $\pi_{(r,n,d_\bullet)}^{-1}(F)$. The monodromy group of $\pi_{(r,n,d_\bullet)}$ acts transitively on the fiber $\pi_{(r,n,d_\bullet)}^{-1}(F)$ since the incidence variety $\Gamma$ is irreducible, while higher transitivity is equivalent to irreducibility of fiber products of the incidence variety with itself \cite{ngalois,GGEGA}. More detail about monodromy groups can be found in \cite{Hatcher}. 

\begin{definition}
The \defcolor{Galois group $\gal_{(r,n,d_\bullet)}$} of the Fano problem $(r,n,d_\bullet)$ is the monodromy group of the restriction of $\pi_{(r,n,d_\bullet)}$ to a covering space. 
\end{definition}

These Galois groups were first defined algebraically by Jordan \cite{Jordan}. The map $\pi_{(r,n,d_\bullet)}:\Gamma\to\mathbb{C}^{(r,n,d_\bullet)}$ induces a reverse inclusion of the function fields of these varieties $\mathbb{C}(\mathbb{C}^{(r,n,d_\bullet)})\hookrightarrow\mathbb{C}(\Gamma)$. This expresses $\mathbb{C}(\Gamma)$ as an algebraic extension of $\mathbb{C}(\mathbb{C}^{(r,n,d_\bullet)})$ of degree $\deg(r,n,d_\bullet)$. The Galois group of the Fano problem $(r,n,d_\bullet)$ defined by Jordan is the Galois group of the normal closure of this field extension. The equivalence of the geometric definition with this algebraic definition was shown by Harris \cite{Harris}, but traces back to Hermite \cite{Hermite}. 

Jordan considered the problem of lines in $\mathbb{P}^3$ on a cubic surface and its Galois group, $\gal_{(1,3,(3))}$. He observed that the Galois group acting on a given Fano scheme must preserve the incidence structure of the lines, and so is a subgroup of the Coxeter group $E_6$. Harris later proved Jordan's inclusion to be an equality, $\gal_{(1,3,(3))} = E_6$, and studied a generalization of this problem. He showed that for $n\ge 4$, the Fano problem of lines in $\mathbb{P}^n$ on a hypersurface of degree $2n-3$ has Galois group equal to the symmetric group \cite{Harris}, such a Galois group is called \defcolor{fully symmetric}. To prove his result, Harris observed for $n\ge 4$ the Galois groups $\gal_{(1,n,(2n-3))}$ are highly transitive and observed systems $F\in\mathbb{C}^{(1,n,(2n-3))}$ whose local monodromy generates a simple transposition.

Hashimoto and Kadets later took up the study of these Galois groups more generally. They studied the Fano problem of $r$--planes in $\mathbb{P}^{2r+2}$ on the intersection of two quadrics for $r\ge 1$ and determined the Galois group to be the Coxeter group $\gal_{(r,2r+2,(2,2))} = D_{2r+3}$. It was then shown that the Fano problems of lines in $\mathbb{P}^3$ on a cubic surface, $(1,3,(3))$, and $r$--planes in $\mathbb{P}^{2r+2}$ on the intersection of two quadric hypersurfaces, $(r,2r+2,(2,2))$ for $r\ge 1$ are special in the sense that these are the only Fano problems where the $r$--planes of a general Fano scheme intersect. Further, Hashimoto and Kadets were able to prove claims about those Fano problems $(r,n,d_\bullet)$ not equal to $(1,3,(3))$ or $(r,2r+2,(2,2))$ for $r\ge 1$. These Galois groups were shown to be highly transitive, and classical results from group theory were used to prove the following.

\begin{theorem}[Hashimoto, Kadets]
If $(r,n,d_\bullet)$ is a Fano problem not equal to $(1,3,(3))$ or $(r,2r+2,(2,2))$ for some $r\ge 1$, then $\gal_{(r,n,d_\bullet)}$ contains the alternating group.
\end{theorem}

Such a Galois group with this property are said to be \defcolor{at least alternating}. The result above states that the only Galois groups of Fano problems which are not at least alternating occur in the special cases of lines in $\mathbb{P}^3$ on a cubic surface and $r$--planes in $\mathbb{P}^{2r+r}$ on the intersection of two quadrics. A complete classification of Galois groups of Fano problems rests on determining whether each at least alternating Galois group of a Fano problem is the alternating group or the symmetric group. The rest of this paper is devoted to proving that many of these at least alternating Galois groups of Fano problems are in fact fully symmetric. We use computational methods to extend Harris' technique of producing a simple transposition in the Galois group, proving they are fully symmetric.

%%%%%%%%%%%%%%%%%%%%%%%%%
%%Computational Methods%%
%%%%%%%%%%%%%%%%%%%%%%%%%
\section{Computational Methods}
Harris showed the Galois group $\gal_{(1,n,(2n-3))}$ contains a simple transposition by producing an instance $F\in\mathbb{C}^{(1,n,(2n-3))}$ with the property that the fiber $\pi_{(1,n,(2n-3))}^{-1}(F)=\mathcal{V}_r(F)$ consists of a unique double point and $\deg(1,n,(2n-3))-2$ smooth points. Since $\Gamma$ and $\mathbb{C}^{(1,n,(2n-3))}$ are smooth, irreducible varieties, the local monodromy around such a point generates a simple transposition as described by Harris \cite{Harris}. 

We describe symbolic and numerical methods of verifying these conditions for a given $F\in\mathbb{C}^{(r,n,d_\bullet)}$ for more general Fano problems $(r,n,d_\bullet)$. The tools we use to verify these conditions are applicable to polynomial systems. We briefly detail our choice of local coordinates on the Grassmanian and how we describe a general Fano scheme as the zeros of a polynomial system in these coordinates.

%conversion to polynomials
We use local coordinates determined by the classical Pl\"{u}cker embedding of the Grassmanian, constructed as follows. An $r$--plane in $\mathbb{P}^n$ is the projectivization of an $(r+1)$--plane in $\mathbb{C}^{n+1}$ and can be represented as the column span of a full rank $(n+1)\times(r+1)$ matrix. The Grassmanian $\mathbb{G}(r,\mathbb{P}^n)$ is then the quotient of the space of full rank $(n+1)\times(r+1)$ matrices by the right action of invertible $(r+1)\times(r+1)$ matrices. This right action uniformly scales the maximal minors of a given $(n+1)\times(r+1)$ matrix and the map sending an $r$--plane to the vector of maximal minors of a matrix representing it defines the projective embedding. Now consider the Zariski open subset $U\subseteq\mathbb{G}(r,\mathbb{P}^n)$ of $r$--planes represented by a matrix whose bottom $(r+1)\times(r+1)$ minor is nonzero. Those $r$--planes in $U$ can be uniquely represented by a matrix whose bottom $(r+1)\times(r+1)$ submatrix is the identity matrix, and the entries of the top $(n-r)\times(r+1)$ submatrix give local coordinates for $U$. 

Fix a Fano problem $(r,n,d_\bullet)$ and $F=(f_1,\dotsc,f_s)\in\mathbb{C}^{(r,n,d_\bullet)}$. An $r$--plane $\ell\in U$ is parameterized as the column span of the matrix representing it. Substituting this parameterization into $f_i$, the restriction $f_i|_\ell$ is a degree $d_i$ form in $r+1$ parameters whose vanishing amounts to the vanishing of its $\left(\begin{smallmatrix}d_i+r\\r\end{smallmatrix}\right)$ coefficients. These coefficients are polynomials in the local coordinates for $U$, totaling a system of $\sum_{i=1}^s\left(\begin{smallmatrix}d_i+r\\r\end{smallmatrix}\right)=(n-r)(r+1)$ polynomials in $(n-r)(r+1)$ variables. That is, The $r$--planes $\ell\in U\cap\mathcal{V}_r(F)$ are described by the zero set of a square polynomial system. For general $F$, $\mathcal{V}_r(F)\subseteq U$ and the Fano scheme $\mathcal{V}_r(F)$ is completely described by this system. 

For various Fano problems $(r,n,d_\bullet)$, we will construct $F\in\mathbb{C}^{(r,n,d_\bullet)}$ with the property that the resulting system of polynomials describing the Fano scheme $\mathcal{V}_r(F)$ has an obvious candidate for a double point. We describe a purely symbolic computation to verify that this candidate is in fact a double point, and detail the use of numerical certification to show there are $\deg(r,n,d_\bullet)-2$ smooth zeros.

%simple double roots
\subsection*{Simple Double Roots}
Let $G$ be a square polynomial system and denote its Jacobian and Hessian by $DG$ and $D^2G$ respectively. A point $x\in\mathbb{C}^m$ is a \defcolor{simple double zero} of $G$ if $G(x)=0$, $\ker DG(x)$ is spanned by a single non-zero vector $v\in\mathbb{C}^m$, and
\begin{align*}
D^2G(x)(v,v)\not\in\text{Im}\,DG(x).
\end{align*}
Dedieu and Shub studied simple double zeros of square polynomial systems in efforts to extend previous work of Smale. In their work, they show that simple double zeros have multiplicity two and they compute a positive separation bound for simple double zeros of a system from other zeros of the same system, showing that simple double zeros are isolated zeros of multiplicity two \cite{DedieuShub}. When the system $G$, point $x$, and tangent vector $v$ have complex rational coefficients and coordinates, these conditions can be checked with a symbolic computation.

\subsection*{Interval Arithmetic}
One flavor of numerical certification comes from interval arithmetic. A complex interval is the set of tuples of complex numbers such that each the real and imaginary parts of each coordinate lies in an interval. The setwise sum and difference of two complex intervals is again a complex interval, however the setwise product of two complex intervals need not be. Usual arithmetic operations are defined on the space of complex intervals as to contain their setwise counterpart. This allows for the evaluation of polynomials and other functions at complex intervals, the result of which contains the setwise evaluation of the function. More information on complex intervals and their arithmetic can be found in \cite{Mayer}.

The Krawczyk operator of $G$ given a point $x$ and an invertible $m\times m$ matrix $Y$, $K_{G,x,Y}$, is a generalization of the Newton operator that acts on the space of complex intervals. A result due to Rump allows us isolate zeros to $G$ using the Krawczyk operator \cite{Rump}. Older versions of this theorem exist for real intervals, and more refined theorems exist for complex intervals \cite{Moore,BRT}.

\begin{theorem}[Rump]
Let $G$ be a square polynomial system in $m$ variables, $x\in\mathbb{C}^m$ a point, and $Y$ an invertible $m\times m$ matrix. If $I$ is a complex interval such that
\begin{align*}
K_{G,x,Y}(I)\subseteq I,
\end{align*}
then $I$ contains a zero of $G$.
\end{theorem}

These theorems from interval arithmetic allow one to compute bounding sets on zeros of a system $G$. As we will soon see, in our setting isolating these zeros will be sufficient for proving they are smooth.

A benefit of using interval arithmetic over other methods of numerical certification is that the conditions can be verified with floating-point arithmetic through the use of proper rounding etiquitte. This computational ease drastically decreases the time required to compute and certify bounding sets for zeros of a system as described in \cite{BRT}. Numerical certification using inverval arithmetic has been implemented in the \texttt{julia} package \texttt{HomotopyContinuation.jl} \cite{HCJL}, which we make great use of. Given a system $G$, this software will both compute approximate zeros of $G$ and attempt to provide an interval satisfying the condition above for each approximate solution.

%%%%%%%%%%%
%%Results%%
%%%%%%%%%%%
\section{Results}
\begin{theorem}
Those Fano problems at least alternating Galois group and degree less than 75,000 have fully symmetric Galois group. More precisely, those Fano problems not equal to $(1,3,(3))$ or $(r,2r+2,(2,2))$ for $r\ge 1$ and with degree less than 75,000 have Galois group equal to the symmetric group.
\end{theorem}

\begin{proof}
Let $(r,n,d_\bullet)$ be a Fano problem not equal to $(1,3,(3))$ or $(r,2r+2,(2,2))$ for $r\ge 1$ (and having degree less than 75,000). As a means of finding $F\in\mathbb{C}^{(r,n,d_\bullet)}$ whose Fano scheme has a double point and $\deg(r,n,d_\bullet)-2$ smooth points, prescribe a subscheme of $\mathcal{V}_r(F)$. Choose $\ell\in\mathbb{G}(r,\mathbb{P}^n)$ and a tangent vector $v\in T_\ell\mathbb{G}(r,\mathbb{P}^n)$ to lie in the Fano scheme $\mathcal{V}_r(F)$. Both of these conditions are linear constraints in $\mathbb{C}^{(r,n,d_\bullet)}$, so $F$ may be chosen to have complex rational coefficients if $\ell$ and $v$ are chosen to have complex rational local coefficients. We choose $F$ to be general satisfying these conditions while also having complex rational coefficients. We wish to show this $F$ has a Fano scheme consisting of a double point and $\deg(r,n,d_\bullet)-2$ smooth points.

Write the square polynomial system describing the Fano scheme $\mathcal{V}_r(F)$ in local coordinates by $G$. As the coefficients of $G$ depend linearly on the coefficients of $F$, the coefficients of $G$ are complex rational as well and we may apply the techniques described above to verify whether $F$ satisfies the desired properties.

The $r$--plane $\ell\in\mathcal{V}_r(F)$ in local coordinates, written $x_\ell$, is a zero of the system $G$. As the point $\ell$ has some tangency in the Fano scheme $\mathcal{V}_r(F)$, it is our candidate to be the double point of $\mathcal{V}_r(F)$. As $G$, $\ell$, and $v$ have all been chosen to have complex rational coefficients and coordinates, it is symbolically checked that $x_\ell$ is in fact a simple double zero of $G$ and hence, $\ell$ is a double point of $\mathcal{V}_r(F)$.

Approximate solutions for the remaining zeros to $G$ are readily computed by softwares such as \texttt{NAG4M2} \cite{NAG4M2}, \texttt{Bertini} \cite{Bertini}, and \texttt{HomotopyContinuation.jl} \cite{HCJL} and numerical certification techniques are used to isolate the remaining $\deg(r,n,d_\bullet)-2$ zeros of $G$ with disjoint sets, each of which does not contain $x_\ell$. As there are at most $\deg(r,n,d_\bullet)$ many isolated solutions to $G$ counting multiplicity, isolating these $\deg(r,n,d_\bullet)-2$ solutions (and $x_\ell$) also verifies that these isolated zeros are smooth. 

This shows $F\in\mathbb{C}^{(r,n,d_\bullet)}$ is such that the Fano scheme $\mathcal{V}_r(F)$ consists of a double point and $\deg(r,n,d_\bullet)-2$ smooth points, and the local monodromy around $F$ generates a simple transposition. As the Galois group $\gal_{(r,n,d_\bullet)}$ is either the alternating group or the symmetric group, it follows $\gal_{(r,n,d_\bullet)}$ is the symmetric group.
\end{proof}

This amounts to proving 12 at least alternating Galois groups of Fano problems are in fact fully symmetric, which were previously unknwon. Table \ref{Big Fano} lists the 12 problems whose Galois group was shown to be the symmetric group as well as timings indicating the time (in seconds (s)) needed to perform the numerical certification using the software \texttt{HomotopyContinuation.jl}. For each of these problems, the computed data of $F\in\mathbb{C}^{(r,n,d_\bullet)}$, the double point $\ell\in\mathbb{G}(r,\mathbb{P}^n)$, the tangent vector $v\in T_\ell\mathcal{V}_r(F)$, and sets isolating for each of the smooth points is available at \cite{GithubRepo}. This further includes code to verify the data satisfies the desired properties and code to generate data for larger Fano problems as well. This result suggests that all at least alternating Galois groups of Fano problems are fully symmetric, though this problem is still open.

\begin{table}[htb]
  \caption{Large Fano Problems}
  \label{Big Fano}
  \def\arraystretch{1.2}
  \begin{tabular}{||c|c|c|c|c|c||}
    \hline
    $~r~$ & $~n~$ & $~d_\bullet~$ & $~\deg(r,n,d_\bullet)~$ & $~\texttt{HomotopyCon}$ (s)~\\
    \hline\hline
    1 & 7 & $(2,2,2,2)$ & 512 & .61\\
    \hline
    1 & 6 & $(2,2,3)$ & 720  & .87\\
    \hline
    2 & 8 & $(2,2,2)$ & 1024 & 1.57\\
    \hline
    1 & 5 & (3,3) & 1053 & .32\\
    \hline
    1 & 5 & (2,4) & 1280 & .73\\
    \hline
    1 & 10 & (2,2,2,2,2,2) & 20480 & 15.44\\
    \hline
    1 & 9 & (2,2,2,2,3) & 27648 & 25.97\\
    \hline
    2 & 10 & (2,2,2,2) & 32768 & 36.67\\
    \hline
    1 & 8 & (2,2,3,3) & 37584 & 38.23\\
    \hline
    1 & 8 & (2,2,2,4) & 47104 & 111.88\\
    \hline
    1 & 7 & (3,3,3) & 51759 & 42.86\\
    \hline
    1 & 7 & (2,3,4) & 64512 & 125.63 \\
    \hline
  \end{tabular}
\end{table}

\bibliographystyle{abbrv}
\bibliography{CGGFFP}

\end{document}